\documentclass[12pt, twoside]{article}
\usepackage[margin=1.5in]{geometry}
\usepackage{amsmath}
\usepackage{amsfonts}
\usepackage{amsthm}
\usepackage{graphicx}
\usepackage{amssymb}
\usepackage[dvips]{epsfig}
\usepackage{subfig}
\usepackage{epstopdf}
\usepackage[hang,flushmargin]{footmisc} 
\usepackage{float}
\usepackage{comment}

\newcommand{\R}{\mathbb{R}}
\newcommand{\C}{\mathbb{C}}

\DeclareMathOperator{\sech}{sech}

\numberwithin{equation}{section}

\let \eps \varepsilon

\newtheorem{theorem}{Theorem}

\newtheorem{remark}[theorem]{Remark}

\newtheorem{lemma}[theorem]{Lemma}

\usepackage{fancyhdr}
\begin{document}
\title{A Degenerate Edge Bifurcation in the 1D Linearized Nonlinear Schr\"{o}dinger Equation}

\author{Matt Coles and  Stephen Gustafson\\ \emph{\small{Department of Mathematics, University of British Columbia}} \\ \emph{\small{1984 Mathematics Road, Vancouver, British Columbia, Canada V6T 1Z2}}}

\maketitle

\begin{abstract}
This work deals with the focusing Nonlinear Schr\"{o}dinger Equation in one dimension with pure-power nonlinearity near cubic. We consider the spectrum of the linearized operator about the soliton solution. 
When the nonlinearity is exactly cubic, the linearized operator has resonances at the
edges of the essential spectrum. 
We establish the degenerate bifurcation of these resonances to eigenvalues as the
nonlinearity deviates from cubic.
The leading-order expression for these eigenvalues 
is consistent with previous numerical computations. 
\end{abstract}

\let\thefootnote\relax


\section{Introduction} \label{intro}
The focusing, pure-power, Nonlinear Schr\"{o}dinger Equation 
for $\psi(x,t) \in \C$, $x \in \R^n$, $t \in \R$,
\begin{align*}
i \partial_t \psi = -\Delta \psi - |\psi|^{p-1} \psi
\qquad \qquad (\mbox{NLS}_p) 
\end{align*}
finds applications in quantum mechanics, optics, and other areas, and has seen intensive
mathematical study in recent years (eg. \cite{Sulem, Fibich}). 
(NLS$_p$) famously exhibits {\it solitary waves} (sometimes called {\it solitons}), solutions which maintain a fixed spatial profile, and which are observed to play a key role in the dynamics of general solutions. 
One naturally asks about the {\it stability} of these waves,
which leads immediately to an investigation of the spectrum of the {\it linearized operator} governing the dynamics close to the solitary wave solution. Systematic spectral analysis of the linearized operator 
has a long history (eg. \cite{Wein1, Grillakis}, and for more recent studies 
\cite{Dpel3, Gust, Voug2, Voug3}). 

The principle motivation for the present work comes from \cite{Gust} where  {\it resonance} eigenvalues (with explicit resonance eigenfunctions) were observed to sit at the edges (or {\it thresholds}) of the spectrum for the 1D linearized NLS problem with focusing cubic nonlinearity. Numerically, it was observed that the same problem with power nonlinearity close to $p=3$ (on both sides) has a true eigenvalue close to the threshold.  
In this paper we establish analytically the observed qualitative behaviour.
Stated roughly, our main result is:
\bigskip

{\it
\noindent
for $p \approx 3$, $p \not= 3$, the linearization of the 1D \emph{(NLS$_p$)}
about its soliton has purely imaginary eigenvalues, bifurcating from 
resonances at the edges of the
essential spectrum of linearized \emph{(NLS$_3$)}, 
whose distance from the thresholds is of order $(p-3)^4$.    
}
\bigskip

\noindent
The exact statement is given as Theorem~\ref{main} in 
Section~\ref{secmain}, and includes the precise leading order 
behaviour of the eigenvalues. 


The eigenvalues obtained here, being on the imaginary axis, correspond to {\it stable} behaviour at the linear level.
A further motivation for obtaining detailed information about the spectra of linearized operators is that 
such information is a key ingredient in studying the {\it asymptotic stability} of 
solitary waves:  see \cite{Busl, Cucc1,Cucc2,Gang,Pere,Schl,Bam,Cucc3} for some results of this type.
Such results typically assume the absence of threshold 
eigenvalues or resonances. The presence of a resonance is an exceptional case which complicates the stability analysis by retarding the time-decay of perturbations. 
Nevertheless, the asymptotic stability of solitons in the 1D cubic focusing NLS was recently proved in \cite{Dpel2}. The proof relies on integrable systems technology and so is only available for the cubic equation. 
The solitons are known to be stable in the (weaker) orbital sense for all $p < 5$ (the so-called mass subcritical range) 
while for $p\geq 5$ they are unstable \cite{Gril,Wein2}, but 
the question of  asymptotic stability for $p<5$ and $p \neq 3$ 
seems to be open. The existence (and location) of eigenvalues on the imaginary axis, which is shown here,
should play a role in any attempt on this problem. 


The generic bifurcation of resonances and eigenvalues from the edge of the essential spectrum was studied by \cite{Dpel} and \cite{Voug} in three dimensions.  
Edge bifurcations have also been studied in one dimensional systems using the Evans function in \cite{Evan1} and \cite{Evan2}. 
We do not follow that route, but rather adopt the approach of \cite{Dpel, Voug}
(going back also to \cite{Jens}, and in turn to the classical work \cite{Kato}), 
using a {\it Birman-Schwinger} formulation, resolvent expansion, 
and {\it Lyapunov-Schmidt reduction}. 

Our work is distinct from \cite{Dpel, Voug} due to the unique challenges of working in 
one dimension, in particular the strong singularity of the free resolvent
at zero energy, which among other things necessitates a {\it double} Lyapunov-Schmidt reduction
procedure. 

Moreover, our work is distinct from all of \cite{Evan1,Evan2,Dpel,Voug} in that we study the particular 
(and as it turns out non-generic) resonance and perturbation corresponding to the near-cubic pure-power NLS 
problem. 
Generically, a resonance is associated with the birth or death of an eigenvalue,
and such is the picture obtained in \cite{Dpel, Voug,Evan1, Evan2}:
an eigenvalue approaches the essential spectrum, becomes a resonance on the threshold and then disappears.  
In our setting, the eigenvalue approaches the essential spectrum, sits on the threshold as a resonance, then returns as an eigenvalue. The bifurcation is degenerate in the sense that the expansion of the eigenvalue begins at higher order,
and the analysis we develop to locate this eigenvalue is thus considerably more delicate.



The paper is organized as follows. The problem is set up in Section \ref{secsetup}. In Section \ref{prelim} we collect some results that are necessary for the bifurcation analysis. 
Section \ref{secmain} is devoted to the statement and proof of the main result. 
The positivity of a certain (explicit) coefficient, which is crucial to the proof, is verified numerically; details of this computation are given in 
Section~\ref{secnum}.


\section{Mathematical Setup} \label{secsetup}

\noindent We consider (NLS$_p$) in one space dimension:
\begin{align}\label{NLS}
i \partial_t \psi = - \partial_x^2\psi -|\psi|^{p-1}\psi.
\end{align}
Here $\psi=\psi(x,t): \R \times \R \rightarrow \C$ with $1< p< \infty$.
The NLS (\ref{NLS}) admits solutions of the form 
\begin{align} \label{solitary}
\psi(x,t)=Q_p(x)e^{it}
\end{align}
where $Q_p(x) > 0$ satisfies 
\begin{align}\label{ODE}
-Q_p''-Q_p^{p}+Q_p=0.
\end{align}
In one dimension the explicit solutions 
\begin{align}\label{ground}
  Q^{p-1}_p(x) = \left(\frac{p+1}{2}\right) \sech^2{\left(\frac{p-1}{2}x\right)}
\end{align}
of (\ref{ODE}) for each $p \in (1,\infty)$
are classically known to be the unique $H^1$ solutions of (\ref{ODE}) up to spatial translation and phase rotation (see e.g. \cite{Caze}).
In what follows we study the linearized NLS problem. 
That is, linearize (\ref{NLS}) about the \emph{solitary wave} 
solutions~\eqref{solitary} by considering solutions of the form 
\begin{align*}
\psi(x,t)=\left(Q_p(x) +h(x,t)\right)e^{it}.
\end{align*}
Then $h$ solves, to leading order (i.e. neglecting terms nonlinear in $h$)
\begin{align*}
i \partial_t h = (-\partial_x^2 +1)h-Q_p^{p-1}h -(p-1)Q_p^{p-1}\text{Re}(h).
\end{align*}
We write the above as a matrix equation 
\begin{align*}
  \partial_t \vec{h} = J \hat{H} \vec{h}
\end{align*}
with
\begin{align*}
 \vec{h} := \begin{pmatrix} \text{Re} (h) \\ \text{Im} (h) \end{pmatrix} \quad \quad \quad \quad \quad \ \ 
 J^{-1} := \begin{pmatrix} 0 & -1 \\ 1 & 0 \end{pmatrix} \\
\hat{H} := \begin{pmatrix} -\partial_x^2 +1 -pQ_p^{p-1}& 0 \\ 0& -\partial_x^2 +1 -Q_p^{p-1} \end{pmatrix}.
\end{align*}
The above $J \hat{H}$ is the linearized operator as it appears in \cite{Gust}.
We now consider the system rotated 
\begin{align*}
  i \partial_t \vec{h} = iJ \hat{H} \vec{h}
\end{align*}
and find $U$ unitary so that, 
$UiJ\hat{H}U^{*}=\sigma_3 H$, where $\sigma_3 $ is one of the Pauli matrices and with $H$ self-adjoint: 
\begin{align*}
\sigma_3= \begin{pmatrix} 1&0\\0&-1 \end{pmatrix}, \quad \quad U=\frac{1}{\sqrt{2}}\begin{pmatrix} 1&i \\ 1&-i \end{pmatrix},
\end{align*}
\begin{align*}
H= \begin{pmatrix} -\partial_x^2 +1&0 \\ 0&-\partial_x^2 +1 \end{pmatrix} - \frac{1}{2} \begin{pmatrix} p+1&p-1 \\ p-1&p+1 \end{pmatrix} Q^{p-1}_p =: \tilde{H}+V^{(p)}.
\end{align*}
In this way we are consistent with the formulation of \cite{Dpel,Voug}. 
We can also arrive at this system, $i \partial_t \vec{h} = \sigma_3 H \vec{h}$, by letting 
$\vec{h} = \begin{pmatrix} h & \bar{h}  \end{pmatrix}^T$
from the start. 

Thus we are interested in the spectrum of 
\begin{align*}
  \mathcal{L}_p := \sigma_3 H
\end{align*}   
and so in what follows we consider the eigenvalue problem 
\begin{align}\label{spec}
\mathcal{L}_p u = z u, \qquad z \in \C, \qquad u \in L^2(\R,\C^2) .
\end{align}
That the essential spectrum of $\mathcal{L}_p$ is 
\begin{align*}
  \sigma_{ess}(\mathcal{L}_p)=(-\infty,-1] \cup  [1,\infty)
\end{align*} 
and $0$ is an eigenvalue of $\mathcal{L}_p$ are standard facts \cite{Gust}.

When $p=3$ we have the following \emph{resonance} at the threshold $z=1$ 
\cite{Gust}
\begin{align} \label{res}
u_0= \begin{pmatrix} 2-Q^2_3 \\ -Q^2_3 \end{pmatrix} = 2\begin{pmatrix} \tanh^2 x \\ - \sech^2 x \end{pmatrix}
\end{align}
in the sense that
\begin{align} \label{reseq}
  \mathcal{L}_3 u_0 = u_0, \qquad 
  u_0 \in L^{\infty}, \qquad 
  u_0 \notin L^q, \mbox{ for } q< \infty.
\end{align}  
Our main interest is how this resonance bifurcates when $p \neq 3$ but $|p-3|$ is small. 
As is natural we seek an eigenvalue of (\ref{spec}) in the following form 
\begin{align} \label{ev}
z=1-\alpha^2, \qquad \alpha > 0 .
\end{align}
We note that the spectrum of $\mathcal{L}_p$ 
for the soliton (\ref{ground}) may only be located on the Real or Imaginary axes \cite{Gust}, and so any eigenvalues in the neighbourhood of $z=1$ must be real. There is also a resonance at $z=-1$ which we do not mention further; symmetry of the spectrum of $\mathcal{L}_p$ ensures the two resonances  bifurcate in the same way. 

We now recast the problem in accordance with the {\it Birman-Schwinger formulation} (pp. 85 of \cite{Book}), as in~\cite{Dpel,Voug}. 
For~\eqref{ev}, \eqref{spec} becomes 
\begin{align*}
(\sigma_3 \tilde{H} -1+\alpha^2)u &= -\sigma_3 V^{(p)} u.
\end{align*}
The constant-coefficient operator on the left is now invertible so we can write
\begin{align*}
u=-(\sigma_3 \tilde{H} -1 + \alpha^2)^{-1} \sigma_3 V^{(p)} u =: -R^{(\alpha)} V^{(p)} u.
\end{align*}
Set
\begin{align*}
 w := |V_0|^{\frac{1}{2}}u, \quad V_0 := V^{(p=3)}
\end{align*}
and apply $|V_0|^{\frac{1}{2}}$ to arrive at the problem 
\begin{align}\label{B-S}
  w= -K_{\alpha,p}w, \qquad
  K_{\alpha,p} := |V_0|^{\frac{1}{2}} R^{(\alpha)} V^{(p)} |V_0|^{-\frac{1}{2}}
\end{align}
with 
\begin{align} \label{resolvent}
R^{(\alpha)}= \begin{pmatrix} (-\partial_x^2 + \alpha^2)^{-1} & 0 \\ 0 & (-\partial_x^2+2 -\alpha^2)^{-1} \end{pmatrix}.
\end{align}
We now seek solutions $(\alpha,w)$ of (\ref{B-S}) which correspond to eigenvalues $1-\alpha^2$ and eigenfunctions $|V_0|^{-\frac{1}{2}} w$ of (\ref{spec}). 
The decay of the potential $V^{(p)}$ and hence $|V_0|^\frac{1}{2}$ now allows us to work in the space $L^2=L^2(\R,\C^2)$, whose standard inner product we denote by $\langle\cdot,\cdot\rangle$. 

The resolvent $R^{(\alpha)}$ has integral kernel 
\begin{align*}
R^{(\alpha)}(x,y)=\begin{pmatrix} \frac{1}{2\alpha} e^{-\alpha|x-y|} & 0 \\ 0 
& \frac{1}{2\sqrt{2-\alpha^2}} e^{-\sqrt{2-\alpha^2}|x-y|} \end{pmatrix}
\end{align*}
for $\alpha>0$.
We expand $R^{(\alpha)}$ as 
\begin{align} \label{resolvexp}
  R^{(\alpha)} = \frac{1}{\alpha} R_{-1} + R_0 + \alpha R_1 + \alpha^2R_R.
\end{align}     
These operators have the following integral kernels
\begin{align*}
R_{-1}(x,y)= \begin{pmatrix} \frac{1}{2} & 0 \\ 0 & 0 \end{pmatrix}, 
R_0(x,y) = \begin{pmatrix} -\frac{|x-y|}{2} & 0 \\ 0 &  \frac{e^{-\sqrt{2}|x-y|}}{2\sqrt{2}} \end{pmatrix}, 
R_1(x,y) = \begin{pmatrix} \frac{|x-y|^2}{4} & 0 \\ 0 & 0 \end{pmatrix} 
\end{align*}
and for $\alpha > 0$ the remainder term $R_R$ is continuous in $\alpha$ 
and uniformly bounded as an operator from a weighted $L^2$ space (with sufficiently strong polynomial weight) to its dual,
and the entries of $R_R(x,y)$ grow at most quadratically in $|x-y|$. 
We also expand the potential $V^{(p)}$ in $\eps$ where $\eps :=p-3$
\begin{align} \label{potexp}
V^{(p)}= V_0 + \eps V_1 + \eps^2 V_2 + \eps^3 V_R ,
\qquad \eps := p-3
\end{align}
and
\begin{align*}
&V_0 = -\begin{pmatrix} 2 & 1 \\ 1 & 2 \end{pmatrix} Q_3^2 
\quad& V_1= -\frac{1}{2} \begin{pmatrix} 1 & 1 \\ 1 & 1 \end{pmatrix} Q_3^2 - \begin{pmatrix} 2 & 1 \\ 1 & 2 \end{pmatrix} q_1 
\\ &V_2 = -\frac{1}{2} \begin{pmatrix} 1 & 1 \\ 1 & 1 \end{pmatrix} q_1 - \begin{pmatrix} 2 & 1 \\ 1 & 2 \end{pmatrix} q_2 
\quad &V_R = -\frac{1}{2} \begin{pmatrix} 1 & 1 \\ 1 & 1 \end{pmatrix} q_2 - \begin{pmatrix} 2 & 1 \\ 1 & 2 \end{pmatrix} q_R 
\\&|V_0|^{\frac{1}{2}}= \frac{1}{2} \begin{pmatrix} \sqrt{3}+1 & \sqrt{3} -1 \\ \sqrt{3} -1 & \sqrt{3} +1 \end{pmatrix} Q_3. 
\end{align*}
Here we have expanded
\begin{align*}
Q_p^{p-1}(x) = Q_3^2(x) + \eps q_1(x) + \eps^2 q_2(x) + \eps^3 q_R(x)
\end{align*}
and the computation gives
\begin{align*}
Q_3^2(x)&=2\sech^2x,
\quad \quad q_1(x)=\sech^2x\left(\frac{1}{2} - 2x\tanh x \right) \\ 
q_2(x)&=\frac{1}{2} \left( 2x^2 \tanh^2x \sech^2x -x^2 \sech^4x -x \tanh x \sech^2x \right). 
\end{align*}
By Taylor's theorem, the remainder term $q_R(x)$ satisfies an 
estimate of the form $|q_R(x)| \leq C (1+ |x|^3) \sech^2(x/2)$ for some constant $C$ which is uniform in $x$ and $\eps \in (-1,1)$.  
We will henceforth write 
\begin{align*}
Q \mbox{ for } Q_3 \quad  \mbox{ and } \quad K_{\alpha,\eps} \mbox{ for } K_{\alpha,p}. 
\end{align*}

\section{Some Preliminaries}\label{prelim}

We study~\eqref{B-S}, that is:
\begin{align}\label{unexp}
(K_{\alpha,\eps} +1)w=0. 
\end{align}
Using the expansions~\eqref{resolvexp} and~\eqref{potexp}
for $R^{(\alpha)}$ and $V^{(p)}$ we make the following expansion
\begin{equation} \label{Kexp}
\begin{split}
K_{\alpha,\eps}&= \frac{1}{\alpha}\left( K_{-10} + \eps K_{-11}+\eps^2K_{-12} + \eps^3K_{R1} \right) \\
 &\quad + K_{00} + \eps K_{01} + \eps^2 K_{02} + \eps^3 K_{R2} \\
 &\quad + \alpha K_{10} + \alpha \eps K_{R3} \\
 &\quad + \alpha^2 K_{R4}
\end{split}
\end{equation}
where $K_{R4}$ is uniformly bounded and continuous in $\alpha>0$
and $\eps$ in a neighbourhood of $0$, as an operator on $L^2(\R,\C^2)$.
 
Before stating the main theorem we assemble some necessary facts about the above operators. 

\begin{lemma} \label{bdd}
Each operator appearing in the expansion~\eqref{Kexp}
for $K_{\alpha,\eps}$ is a Hilbert-Schmidt (so in particular
bounded and compact) operator from $L^2(\R,\C^2)$ to itself. 
\end{lemma}
\begin{proof}
This is a straightforward consequence of the spatial decay of the weights 
which surround the resolvent. 
The facts that $\| |V_0|^{-\frac{1}{2}} \| \leq C/\sech(x)$, and that 
$|V_0|^{1/2}$ decays like $\sech(x)$, while each of $V_0$, $V_1$, $V_2$, decay like $\sech^2(y)$, and $V_R$ decays at worst
like $\sech^2(3x/4)$ (say if we restrict to $|\eps| < \frac{1}{2}$)
imply easily that these operators all have square integrable integral kernels.
\end{proof}

We will also need the projections $P$ and $\overline{P}$ which are defined as follows: for $f \in L^2$ let
\begin{align*}
Pf:=\frac{\langle v,f\rangle v}{\|v\|^2}, \quad v:=|V_0|^{\frac{1}{2}} \left(
\begin{array}{c}
1\\
0\\
\end{array}
\right)
\end{align*}
as well as the complementary $\overline{P}:=1-P$. 
A direct computation shows that for any $f \in L^2$ we have
\begin{align} \label{genproj}
  K_{-10}f = -4 P f.
\end{align}
Note that all operators in the expansion containing $R_{-1}$ return 
outputs in the direction of $v$. 

\begin{lemma}\label{kernel}
The operator $\overline{P}(K_{00}+1)\overline{P}$ has a one dimensional kernel spanned by 
\begin{align*}
  w_0 := |V_0|^{1/2} u_0
\end{align*} 
as an operator from \emph{Ran}$(\overline{P})$ to \emph{Ran}$(\overline{P})$. 
\end{lemma}
\begin{proof}
First note that by~\eqref{reseq}
\begin{align} \label{reseq2}
  V_0 u_0 = \sigma_3 u_0 - \tilde{H} u_0, \qquad
  \left[ V_0 u_0 \right]_1 = \left[ u_0 \right]_1''
\end{align}
from which it follows that 
\begin{align*}
  P w_0 = 0, \quad \mbox{ i.e. } w_0 \in \text{Ran}(\overline{P}).
\end{align*}   
Then a direct computation using~\eqref{reseq2}, the expansion~\eqref{Kexp},
the expression for $R_0$, and integration by parts, shows that
\begin{align*}
  (K_{00}+1) w_0 = 2 v
\end{align*} 
and so indeed $\overline{P}(K_{00}+1)\overline{P}w_0=0$. 

Theorem 5.2 in \cite{Jens} shows that the kernel of the analogous scalar operator can be at most one dimensional. We will use this argument, adapted to 
the vector structure, to show that any two non-zero elements of the kernel must be multiples of each other. 
Take $w \in L^2$ with $\langle w,v \rangle=0$ and $\overline{P}(K_{00}+1)w=0$. That is $(K_{00}+1)w=cv$ for some constant $c$. This means
\begin{align*}
|V_0|^{\frac{1}{2}} R_0 V_0|V_0|^{-\frac{1}{2}}w + w = c |V_0|^{\frac{1}{2}}
\left(
\begin{array}{c}
1\\
0\\
\end{array}
\right).
\end{align*}
Let $w=|V_0|^{\frac{1}{2}}u$ where $u=
\left(
\begin{array}{c}
u_1\\
u_2\\
\end{array}
\right)$. We then obtain, after rearranging and expanding
\begin{align*}
\left(
\begin{array}{c}
u_1\\
u_2\\
\end{array}
\right)
=
\left(
\begin{array}{c}
c - \frac{1}{2} \int_\R |x-y| Q^2(y) \left( 2u_1(y)+u_2(y) \right) dy \\
\frac{1}{2 \sqrt{2}} \int_\R \exp \left( -\sqrt{2} |x-y|\right) Q^2(y)  (u_1(y) + 2 u_2(y)) dy \\
\end{array}
\right). 
\end{align*}
We now rearrange the first component. Expand 
\begin{align*}
-\frac{1}{2}\int_\R |x-y| Q^2(y) &(2u_1(y)+u_2(y)) dy \\
= &-\frac{1}{2} \int_{-\infty}^x (x-y) Q^2(y) (2u_1(y)+u_2(y)) dy \\
  &- \frac{1}{2} \int_x^\infty (y-x) Q^2(y) (2u_1(y)+u_2(y)) dy
\end{align*}
and rewrite the first term as 
\begin{align*}
&-\frac{x}{2} \int_{-\infty}^x  Q^2(y) (2u_1(y)+u_2(y)) dy + \frac{1}{2} \int_{-\infty}^x  y Q^2(y) (2u_1(y)+u_2(y)) dy \\
&= \frac{x}{2} \int_x^\infty Q^2(y) (2u_1(y)+u_2(y)) dy + b - \frac{1}{2} \int_x^\infty y Q^2(y) (2u_1(y)+u_2(y)) dy
\end{align*}
where
\begin{align*}
b:= \frac{1}{2} \int_\R  y Q^2(y) (2u_1(y)+u_2(y)) dy
\end{align*}
and where we used $\int_\R 2 Q^2 u_1 + Q^2 u_2 =0$ 
since $\langle w,v \rangle=0$. So putting everything back together we see
\begin{align}\label{u1u2}
\left(
\begin{array}{c}
u_1\\
u_2\\
\end{array}
\right)
=
\left(
\begin{array}{c}
c+b +  \int_x^\infty (x-y) Q^2(y) \left( 2u_1(y)+u_2(y) \right) dy \\
\frac{1}{2 \sqrt{2}} \int_\R \exp \left( -\sqrt{2} |x-y|\right) Q^2(y)  (u_1(y) + 2 u_2(y)) dy \\
\end{array}
\right). 
\end{align}

We claim that as $x \rightarrow \infty$
\begin{align*}
\left(
\begin{array}{c}
u_1\\
u_2\\
\end{array}
\right)
\rightarrow
\left(
\begin{array}{c}
c+b\\
0 \\
\end{array}
\right). 
\end{align*}
Observe
\begin{align*}
\bigg|\int_x^\infty (x-y) Q^2(y) \left( 2u_1(y)+u_2(y) \right) dy\bigg| 
&\leq 
\int_x^\infty |y-x| Q^2(y) | 2u_1(y)+u_2(y) | dy \\
& \leq 
\int_x^\infty |y|  Q^2(y) | 2u_1(y)+u_2(y) | dy \\
& \rightarrow 0
\end{align*}
as $x \rightarrow \infty$. Here we have used the fact that $w \in L^2$ implies $Q|2u_1+u_2| \in L^2$ and that $|y|Q \in L^2$. 
As well, in the second component
\begin{align*}
\int_\R e^{  -\sqrt{2} |x-y|} Q^2(y)  &(u_1(y) + 2 u_2(y)) dy \\
= & e^{-\sqrt{2} x} \int_{-\infty}^x e^{\sqrt{2}y} Q^2(y)  (u_1(y) + 2 u_2(y)) dy \\
&+ e^{\sqrt{2} x} \int_{x}^\infty e^{-\sqrt{2}y} Q^2(y)  (u_1(y) + 2 u_2(y)) dy 
\end{align*}
and 
\begin{align*}
\bigg| &e^{-\sqrt{2} x} \int_{-\infty}^x e^{\sqrt{2}y} Q^2(y)  (u_1(y) + 2 u_2(y)) dy \bigg|\\
&\leq 
e^{-\sqrt{2} x} \int_{-\infty}^x e^{\sqrt{2}y} Q^2(y)  |u_1(y) + 2 u_2(y)| dy \\
&\leq
e^{-\sqrt{2} x} \left( \int_{-\infty}^x e^{2\sqrt{2}y} Q^2(y) dy\right)^{1/2}    \left( \int_{-\infty}^x Q^2(y) |u_1(y) + 2 u_2(y)|^2 dy \right)^{1/2} \\
&\leq
C e^{-\sqrt{2} x} \left( \int_{-\infty}^x e^{2\sqrt{2}y} Q^2(y) dy\right)^{1/2}\\
&\leq
C e^{-\sqrt{2} x} \left( \int_{-\infty}^x e^{2\sqrt{2}y} e^{-2y} dy\right)^{1/2}\\
&\leq
C e^{-\sqrt{2} x} \left( e^{-2\sqrt{2}x}e^{-2x} \right)^{1/2} \leq C e^{-x} \rightarrow 0, \quad x \rightarrow \infty
\end{align*}
where we again used $Q|u_1+2u_2| \in L^2$. Similarly,
\begin{align*}
\bigg| e^{\sqrt{2} x} \int_{x}^\infty e^{-\sqrt{2}y} Q^2(y)  (u_1(y) + 2 u_2(y)) dy \bigg| \rightarrow 0
\end{align*}
as $x \rightarrow \infty$ which addresses the claim.

Next we claim that if $c+b=0$ in (\ref{u1u2}) then $u \equiv 0$. To address the claim we first note that if $c+b=0$ then $u \equiv 0$ for all $x \geq X$ for some $X$, by estimates similar to those just done. 
Finally, we appeal to ODE theory. Differentiating (\ref{u1u2}) in $x$ twice returns the system
\begin{align}
 u_1'' &= -2Q^2 u_1 - Q^2 u_2 \label{syst1}\\
 u_2'' -2u_2 &= -Q^2 u_1 - 2Q^2 u_2. \label{syst2}
\end{align}
Any solution $u$ to the above with $u \equiv 0$ for all large enough $x$ must be identically zero.

With the claim in hand we finish the argument. Given two non-zero elements of the kernel, say $u$ and $\tilde{u}$ with limits as $x \to \infty$  (written as above)
$c+b$ and $\tilde{c}+\tilde{b}$ respectively, the combination 
\begin{align*}
u^*=u- \frac{c+b}{\tilde{c}+\tilde{b}} \tilde{u} 
\end{align*}
satisfies (\ref{u1u2}) but with $u^*(x) \to 0$ as $x \to \infty$, 
and so $u^* \equiv 0$. 
Therefore, $u$ and $\tilde{u}$ are linearly dependent, as required.  
\end{proof}
\begin{remark} \label{truerem}
Arguments similar to the estimates in the above Lemma \ref{kernel} show that 
for $\alpha > 0$ and $w \in L^2$ solving \emph{(\ref{B-S})} the corresponding eigenfunction of \emph{(\ref{spec})} $u=|V_0|^{-\frac{1}{2}}w$ is in $L^2$ and so the eigenvalue $z=1-\alpha^2$ is in fact a true eigenvalue. 
\end{remark} 
 
Note that $K_{00}$, and hence $\overline{P}(K_{00}+1)\overline{P}$, is self-adjoint. Indeed, a direct computation shows that $V_0=-|V_0|$
and so 
\begin{align*}
K_{00}&=|V_0|^{\frac{1}{2}} R_0 V_0 |V_0|^{-\frac{1}{2}} \\
&=-|V_0|^{\frac{1}{2}} R_0 |V_0|^\frac{1}{2} \\
&=|V_0|^{-\frac{1}{2}}V_0 R_0 |V_0|^\frac{1}{2} \\
&=(K_{00})^{*}.
\end{align*}
As we have seen above in Lemma \ref{bdd}, thanks to the decay of the potential, $\overline{P} K_{00} \overline{P}$ is a compact operator. Therefore, the simple eigenvalue $-1$ of $\overline{P} K_{00} \overline{P}$ is isolated and so 
\begin{align} \label{invert}
  ( \overline{P}(K_{00}+1)\overline{P} )^{-1}
  :  \{ v, \; w_0 \}^{\perp} \to  \{ v, \; w_0 \}^{\perp}
\end{align}  
exists and is bounded. 

Finally, with the above facts assembled, we are now in a position to state the main theorem.

\section{Bifurcation Analysis}\label{secmain}

This section is devoted to the proof of the main result: 
\begin{theorem} \label{main}
There exists $\eps_0>0$ such that for $-\eps_0 \leq \eps \leq \eps_0$ with $\eps \neq 0$ the eigenvalue problem \emph{(\ref{unexp})} has a solution $(\alpha,w)$ 
of the form
\begin{equation}\label{expansion}
  \begin{aligned}
	w&=w_0+\eps w_1+\eps^2w_2 + \tilde{w} \\
	\alpha&=\eps^2 \alpha_2 + \tilde{\alpha}
  \end{aligned}
\end{equation}
where $\alpha_2 > 0$, $w_0$, $w_1$, $w_2$ are known (given below),
and $|\tilde{\alpha}|<C|\eps|^3$ and $\|\tilde{w}\|_{L^2}<C|\eps|^3$ for some $C>0$. 
\end{theorem}

\begin{remark}
This theorem confirms the behaviour observed numerically in \cite{Gust}: for $p\neq 3$ but close to $3$, the linearized operator
$J \hat{H}$ (which is unitarily equivalent to $i {\cal L}_p$) has true, purely imaginary eigenvalues in the gap between the branches of essential spectrum, which approach the thresholds as $p \rightarrow 3$. Note Remark \ref{truerem} to see that $u = |V_0|^{-\frac{1}{2}}w$ is a true $L^2$ eigenfunction of \emph{(\ref{spec})}. 
In addition, the eigenfunction approaches the resonance eigenfunction in some weighted $L^2$ space. 
Furthermore, we have found that $\alpha^2$, the distance of the eigenvalues from the thresholds, is to leading order proportional to $(p-3)^4$. 
Finally, note that $\alpha = \eps^2 \alpha_2  + O(\eps^3)$ with $\alpha_2>0$ gives $\alpha>0$ for both $\eps >0$ and $\eps <0$, 
ensuring the eigenvalues appear on {\it both} sides of $p=3$. 
\end{remark}

The quantities in~(\ref{expansion}) are defined as follows:
\begin{align*}
&w_0 := |V_0|^{\frac{1}{2}}u_0 \\
&Pw_1:= \frac{1}{4} K_{-11}w_0\\
&\overline{P}w_1 := -\big( \overline{P}(K_{00}+1) \overline{P}\big)^{-1} \left(\frac{1}{4}\overline{P}K_{00}K_{-11} w_0 + \overline{P}K_{01} w_0\right) \\
&Pw_2 := \frac{1}{4}\big( K_{-11}w_1 + K_{-12}w_0+\alpha_2(K_{00}+1)w_0 \big)\\
&\overline{P}w_2 := -\big( \overline{P}(K_{00}+1) \overline{P}\big)^{-1} \Bigg( \frac{1}{4}\overline{P} K_{00} K_{-11} w_1 +\frac{1}{4}\overline{P} K_{00} K_{-12}w_0   \\ 
& \quad \quad \quad \quad \quad \quad \quad
+\frac{\alpha_2}{4}\overline{P} K_{00}(K_{00}+1) w_0+\overline{P} K_{01} w_1 + \overline{P} K_{02} w_0 + \alpha_2 \overline{P} K_{10} w_0 \Bigg) \\
&\alpha_2 := \frac{-\frac{1}{4}\langle w_0,  K_{00} K_{-11} w_1\rangle  -\frac{1}{4}\langle w_0,  K_{00} K_{-12}w_0\rangle  -\langle w_0,  K_{01} w_1\rangle  - \langle w_0, K_{02} w_0 \rangle }{ \langle w_0,K_{10}w_0\rangle  + \frac{1}{4}\langle w_0,K_{00}(K_{00}+1)w_0 \rangle   }
\end{align*}
\begin{remark}
A numerical computation shows 
\begin{align*}
  \alpha_2 \approx 2.53/8 > 0.
\end{align*} 
Since the positivity of $\alpha_2$ is crucial to the main result,
details of this computation are described in Section~\ref{secnum}.
\end{remark}
Note that the functions on which $\overline{P}(K_{00}+1)\overline{P}$ is being inverted in the expressions for $\overline{P}w_1$ and $\overline{P}w_2$ are orthogonal to both $w_0$ and $v$, and so these quantities are well-defined
by~\eqref{invert}. The identity
\begin{align*}
\langle w_0,  \frac{1}{4}K_{00}K_{-11} w_0 + K_{01} w_0 \rangle =0
\end{align*}
has been verified analytically. It is because of this identity that the 
$O(\eps)$ term is absent in the expansion of $\alpha$ in (\ref{expansion}). 
The fact that 
\begin{align*}
0=
\langle w_0, \frac{1}{4} K_{00} K_{-11} w_1 +\frac{1}{4}K_{00} K_{-12}w_0 +\frac{\alpha_2}{4}K_{00}(K_{00}+&1) w_0  +K_{01} w_1 \\
&+ K_{02} w_0 + \alpha_2 K_{10} w_0\rangle
\end{align*}
comes from our definition of $\alpha_2$. 

The above definitions, along with~\eqref{genproj}, 
imply the relationships below 
\begin{align} 
0&= K_{-10} w_0 \label{Pw0} \\
0&= K_{-11} w_0 + K_{-10} w_1 \label{Pw1} \\
0&= K_{-10} w_2 +  K_{-11} w_1 + K_{-12} w_0+ \alpha_2 (K_{00} +1) w_0 \label{Pw2} \\
0&=\overline{P}(K_{00}+1)w_1+\overline{P}K_{01}w_0 \label{O1} \\
0&=\overline{P}(K_{00}+1)w_2+\overline{P}K_{01}w_1 + \overline{P}K_{02}w_0 + \alpha_2 \overline{P}K_{10}w_0 \label{O2}
\end{align}
which we will use in what follows. 

Using the expression for $\alpha$ in~\eqref{expansion}, our expansion~\eqref{Kexp} for $K_{\alpha,\eps}$ now takes the form 
\begin{align*}
K_{\alpha,\eps}&= \frac{1}{\alpha}\left( K_{-10} + \eps K_{-11}+\eps^2K_{-12} + \eps^3K_{R1} \right) \\
 &\quad + K_{00} + \eps K_{01} + \eps^2 K_{02} + \eps^3 K_{R2} \\
 &\quad + (\alpha_2\eps^2+\tilde{\alpha}) K_{10} + (\alpha_2\eps^2+\tilde{\alpha}) \eps K_{R3} + (\alpha_2\eps^2+\tilde{\alpha})^2 K_{R4}\\
& =: \frac{1}{\alpha}\left( K_{-10} + \eps K_{-11}+\eps^2K_{-12} + \eps^3K_{R1} \right) + K_{00} + \eps \overline{K}_1 + \tilde{\alpha} \overline{K}_2
\end{align*}
where $\overline{K}_1$ is a bounded (uniformly in $\eps$) operator depending
on $\eps$ but not $\tilde{\alpha}$, while  $\overline{K}_2$ is a bounded 
(uniformly in $\eps$ and $\tilde{\alpha}$) 
operator depending on both $\eps$ and $\tilde{\alpha}$. 

Further decomposing 
\begin{align*}
  \tilde{w}=\beta v +W, \qquad \langle W,v \rangle = 0,
\end{align*} 
we aim to show existence of a solution with the remainder terms $\tilde{\alpha}$, $\beta$ and $W$ small. We do so via a double Lyapunov-Schmidt reduction.

First substitute (\ref{expansion}) to (\ref{unexp}) and apply the projection 
$\overline{P}$ to obtain 
\begin{equation} \label{Pbar}
\begin{split}
0&=\overline{P}(K_{\alpha,\eps} + 1) w \\
&=\overline{P}(K_{\alpha,\eps}+1)(w_0+ \eps w_1 + \eps^2 w_2 + \beta v + W)\\
&=  \overline{P}(K_{00}+1)w_0+\eps \overline{P}(K_{00}+1)w_1 + \eps \overline{P}K_{01}w_0\\
&\quad + \eps^2  \overline{P} (K_{00}+1)w_2 + \eps^2 \overline{P} K_{01} w_1 + \eps^2 \overline{P} K_{02} w_0 + \eps^2 \alpha_2 \overline{P} K_{10} w_0\\
& \quad + \overline{P}(K_{00}+1)(\beta v+W) + \tilde{\alpha} \overline{P}K_{10}w_0 + \overline{P} \left( \eps \overline{K}_1 + \tilde{\alpha} \overline{K}_2 \right)(\beta v +W) \\
& \quad + \eps^3 \overline{P} \left( K_{R2}w_0 +   K_{02} w_1 +  K_{01}w_2  + \eps K_{02}w_2 + \eps K_{R2}w_1 + \eps^2 K_{R2}w_2 \right)\\
 &\quad + (\alpha_2\eps^2+\tilde{\alpha}) \overline{P}K_{10}(\eps w_1 + \eps^2 w_2) + (\alpha_2\eps^2+\tilde{\alpha}) \eps \overline{P}K_{R3}(w_0+\eps w_1 + \eps^2 w_2) \\
& \quad + (\alpha_2\eps^2+\tilde{\alpha})^2 \overline{P}K_{R4}(w_0+\eps w_1 + \eps^2 w_2).
\end{split}
\end{equation}
Making some cancellations coming from Lemma \ref{kernel}, (\ref{O1}) and (\ref{O2}) leads to 
\begin{align*}
-\overline{P}(K_{00}&+1)\overline{P}W=\\
&  \beta\overline{P}K_{00} v + \tilde{\alpha} \overline{P}K_{10}w_0 + \overline{P}  \left( \eps \overline{K}_1 + \tilde{\alpha} \overline{K}_2 \right) (\beta v +W) \\
& \  + \eps^3 \overline{P} \left( K_{R2}w_0 +   K_{02} w_1 +  K_{01}w_2  + \eps K_{02}w_2 + \eps K_{R2}w_1 + \eps^2 K_{R2}w_2 \right)\\
 & \  + (\alpha_2\eps^2+\tilde{\alpha}) \overline{P}K_{10}(\eps w_1 + \eps^2 w_2) + (\alpha_2\eps^2+\tilde{\alpha}) \eps \overline{P}K_{R3}(w_0+\eps w_1 + \eps^2 w_2) \\
&  \   + (\alpha_2\eps^2+\tilde{\alpha})^2 \overline{P}K_{R4}(w_0+\eps w_1 + \eps^2 w_2)\\
&=: \mathcal{F}(W;\eps, \tilde{\alpha},\beta).	
\end{align*}

According to~\eqref{invert}, inversion of  $\overline{P}(K_{00}+1)\overline{P}$ on $\mathcal{F}$ requires the solvability condition
\begin{align} \label{P0}
   P_0 \mathcal{F} = 0, \qquad 
   P_0 := \frac{1}{\| w_0 \|_2^2} \langle w_0, \; \cdot \rangle w_0, 
   \quad \overline{P}_0 := 1 - P_0
\end{align}
which we solve together with the fixed point problem
\begin{align}\label{defW}
  W=\left(-\overline{P}(K_{00}+1)\overline{P}\right)^{-1} 
  \overline{P}_0 \mathcal{F}(W;\eps, \tilde{\alpha},\beta) 
  =: \mathcal{G}(W;\eps, \tilde{\alpha},\beta)
\end{align}
in order to solve~\eqref{Pbar}.

Write
\begin{align*}
\mathcal{F}:=\overline{P}\left( \beta K_{00}v + \tilde{\alpha} K_{10}w_0  + \left( \eps \overline{K}_1 + \tilde{\alpha} \overline{K}_2 \right)(\beta v +W) + \eps^3 f_1 + \eps \tilde{\alpha} f_2 + \tilde{\alpha}^2 h_1   \right)
\end{align*} 
where $f_1$ and $f_2$ denote functions depending on (and $L^2$ bounded uniformly in) $\eps$ but not $\tilde{\alpha}$, while $h_1$ denotes an $L^2$ function depending on (and uniformly $L^2$ bounded in) both $\eps$ and $\tilde{\alpha}$. 
\begin{lemma}\label{W}
For any $M> 0$ there exists $\eps_0> 0$ and $R> 0$ such that for all $-\eps_0 \leq \eps\leq \eps_0$ \emph{(}but $\eps \neq 0$\emph{)} and for all $\tilde{\alpha}$ and $\beta$ with $|\tilde{\alpha}|\leq M |\eps|^3$ and $|\beta|\leq M |\eps|^3$ there exists a unique solution $W \in L^2 \cap \{ v, \; w_0 \}^{\perp}$ of 
$(\ref{defW})$ satisfying $\|W\|_{L^2}\leq R |\eps|^3$. 
\end{lemma}
\begin{proof}
We prove this by means of Banach Fixed Point Theorem. We must show that $\mathcal{G}(W)$ maps the closed ball of radius $R |\eps|^3$ into itself and that $\mathcal{G}(W)$ is a contraction mapping. Taking $W \in L^2$ orthogonal to $v$ and $w_0$ such that $\|W\|_{L^2} \leq R |\eps|^3$ and given $M> 0$ where $|\tilde{\alpha}|\leq M |\eps|^3$ and $|\beta|\leq M |\eps|^3$, we have,
using the boundedness of 
$\left(-\overline{P}(K_{00}+1)\overline{P}\right)^{-1}  \overline{P}_0$,
\begin{align*}
\|\mathcal{G}&\|_{L^2}   \\
&\leq C |\beta|  \|\overline{P}K_{00}v + \overline{P} \left( \eps \overline{K}_1 + \tilde{\alpha} \overline{K}_2 \right)v \|_{L^2}+ C |\tilde{\alpha}|  \|\overline{P}\left( K_{10}w_0 + \eps f_2 + \tilde{\alpha} h_1 \right)\|_{L^2}  \\
& \quad + C \|\overline{P}  \left( \eps \overline{K}_1 + \tilde{\alpha} \overline{K}_2 \right)W\|_{L^2} 
+ |\eps|^3 C \| \overline{P}f_1\|_{L^2} \\
& \leq  C M|\eps|^3  + C M |\eps|^3  + C |\eps|    \|W\|_{L^2} + C |\tilde{\alpha}|   \|W\|_{L^2}+ C |\eps|^3  \\
&  \leq C |\eps|^3  +  C R |\eps|^4  \\
&\leq R |\eps|^3
\end{align*}
for some appropriately chosen $R$ with $|\eps|$ small enough. Here $C$ is a positive, finite constant whose value changes at each appearance. 
Next consider 
\begin{align*}
\| \mathcal{G}(W_1&)-\mathcal{G}(W_2) \|_{L^2}  \\
& \leq  C  \| \overline{P} \left( \eps \overline{K}_1 + \tilde{\alpha} \overline{K}_2 \right)\|_{L^2\rightarrow L^2}   \|W_1-W_2\|_{L^2}\\
& \leq C |\eps|\|\overline{P} \  \overline{K}_{1} \|_{L^2\rightarrow L^2}  \|W_1-W_2\|_{L^2} + C |\tilde{\alpha}|  \| \overline{P} \  \overline{K}_{2} \|_{L^2\rightarrow L^2}  \|W_1-W_2\|_{L^2}\\
&\leq C |\eps|  \|W_1-W_2\|_{L^2} \leq \kappa \|W_1-W_2\|_{L^2}
\end{align*}
with $0<\kappa< 1$ by taking $|\eps|$ sufficiently small. 
Hence $\mathcal{G}(W)$ is a contraction, and we obtain the desired result. 
\end{proof}

Lemma~\ref{W} provides $W$ as a function of $\tilde{\alpha}$ and $\beta$,
which we may then substitute into~\eqref{P0} to get
\begin{align}\label{alphabeta1}
0&=\langle w_0,\mathcal{F}\rangle  \nonumber \\ 
&= \beta\langle w_0,K_{00}v\rangle  + \tilde{\alpha} \langle w_0,K_{10}w_0\rangle + \eps \beta \langle w_0,\overline{K}_1 v\rangle  + \tilde{\alpha} \beta \langle w_0,\overline{K}_2 v\rangle    \nonumber \\
&  \quad+ \eps^3\langle w_0,f_1\rangle  +\eps \tilde{\alpha}\langle w_0,f_2\rangle  
+\tilde{\alpha}^2\langle w_0, h_1 \rangle  + \eps \langle w_0,\overline{K}_1 W\rangle  +  \tilde{\alpha}\langle w_0,\overline{K}_2W\rangle  \nonumber  \\
&=:\beta\langle w_0,K_{00}v\rangle + \tilde{\alpha} \langle w_0,K_{10}w_0\rangle  + \mathcal{F}_1
\end{align}
which is the first of two equations relating $\tilde{\alpha}$ and $\beta$. 

The second equation is the complementary one to~\eqref{Pbar}: substitute (\ref{expansion}) to (\ref{unexp}) but this time multiply by $\alpha$ and take projection $P$ to see
\begin{equation} \label{P}
\begin{split}
0 &= \alpha P(K_{\alpha,\eps} + 1) w \\
&= K_{-10}w_0 + \eps(K_{-11}w_0+K_{-10}w_1) \\
&\quad + \eps^2\left(K_{-10}w_2+K_{-11}w_1+K_{-12}w_0\right) 
+\eps^2\alpha_2 (K_{00}+1)w_0 \\
&\quad  + \eps^3(K_{-11}w_2 + K_{-12}w_1 + K_{R1}w_0  + \eps K_{-12}w_2 + \eps K_{R1}w_1 + \eps^2K_{R1}w_2) \\
& \quad + \beta K_{-10} v + K_{-10}W + \eps (K_{-11} + \eps K_{-12} + \eps^2 K_{R1})(\beta v+W) \\ 
& \quad + \tilde{\alpha} (K_{00}+1) w_0 + \eps^3 \alpha_2 P(K_{00}+1)(w_1 + \eps w_2) +\eps \tilde{\alpha}P(K_{00}+1)(w_1+\eps w_2)\\	
& \quad  + \eps^2 \alpha_2 P(K_{00}+1)(\beta v +W) +  \tilde{\alpha} P(K_{00}+1)(\beta v +W)\\
& \quad + \alpha P(\eps K_{01} + \eps^2 K_{02} + \eps^3 K_{R2}+\alpha K_{10} + \alpha \eps K_{R3} + \alpha^2 K_{R4})  \\
& \quad \quad \quad \quad \quad \quad \quad \quad \quad \quad \quad \quad \quad \quad \quad \quad \quad \quad
 \times (w_0+\eps w_1 + \eps^2 w_2 + \beta v + W).
\end{split}
\end{equation}
After using known information about $w_0, w_1, w_2, \alpha_2$ coming from (\ref{Pw0}), (\ref{Pw1}), (\ref{Pw2}) and noting that $K_{-10}W=-4 P W=0$ from (\ref{genproj}) we have
\begin{align*}
0&=  \beta K_{-10} v + \tilde{\alpha} (K_{00}+1) w_0 \\
&\quad + \eps^3(K_{-11}w_2 + K_{-12}w_1 + K_{R1}w_0  + \eps K_{-12}w_2 + \eps K_{R1}w_1 + \eps^2K_{R1}w_2) \\
& \quad   + \eps (K_{-11} + \eps K_{-12} + \eps^2 K_{R1})(\beta v+W) \\ 
& \quad  + \eps^3 \alpha_2 P(K_{00}+1)(w_1 + \eps w_2) +\eps \tilde{\alpha}P(K_{00}+1)(w_1+\eps w_2)\\	
& \quad  + \eps^2 \alpha_2 P(K_{00}+1)(\beta v +W) +  \tilde{\alpha} P(K_{00}+1)(\beta v +W)\\
& \quad + \alpha P(\eps K_{01} + \eps^2 K_{02} + \eps^3 K_{R2}+\alpha K_{10} + \alpha \eps K_{R3} + \alpha^2 K_{R4})\\
& \quad \quad \quad \quad \quad \quad \quad \quad \quad \quad \quad \quad \quad \quad \quad \quad \quad \quad
 \times (w_0+\eps w_1 + \eps^2 w_2 + \beta v + W).
\end{align*}
Written more compactly, this is 
\begin{align*}
0= &\beta K_{-10}v + \tilde{\alpha}(K_{00}+1) w_0 \\
&\quad + \eps^3 f_4 + \eps \overline{K}_3 (\beta v + W) + \tilde{\alpha} \eps f_5 + \tilde{\alpha} \overline{K}_4 (\beta v +W)+ \tilde{\alpha}^2 h_2
\end{align*}
where $\overline{K}_3$ is a bounded (uniformly in $\eps$) operator containing $\eps$ but not $\tilde{\alpha}$, while $\overline{K}_4$ is a bounded 
(uniformly in $\eps$ and $\tilde{\alpha}$)
operator containing both $\eps$ and $\tilde{\alpha}$. 
Functions $f_4$ and $f_5$ depend on $\eps$ (and are uniformly $L^2$-bounded) but not $\tilde{\alpha}$, while the function 
 $h_2$ depends on both $\eps$ and $\tilde{\alpha}$
 (and is uniformly $L^2$-bounded).
To make the relationship between $\tilde{\alpha}$ and $\beta$ more explicit we take inner product with $v$
\begin{align}\label{alphabeta2}
0&=\beta \langle v, K_{-10}v\rangle + \tilde{\alpha} \langle v, (K_{00}+1) w_0 \rangle + \eps^3 \langle v, f_4 \rangle \nonumber \\
&\quad  + \eps  \langle v, \overline{K}_3 (\beta v + W)\rangle + \tilde{\alpha} \eps \langle v, f_5\rangle + \tilde{\alpha} \langle v, \overline{K}_4 (\beta v +W)\rangle + \tilde{\alpha}^2  \langle v, h_2 \rangle \nonumber \\
&=:\beta \langle v, K_{-10}v\rangle + \tilde{\alpha} \langle v, (K_{00}+1) w_0 \rangle + \mathcal{F}_2. 
\end{align}
Now let 
\begin{align*}
\vec{\zeta}= \begin{pmatrix} \tilde{\alpha} \\ \beta \end{pmatrix} 
\end{align*}
and rewrite (\ref{alphabeta1}) and (\ref{alphabeta2}) in the following way
\begin{align*}
A \vec{\zeta} := \begin{pmatrix} \langle w_0,K_{10}w_0\rangle & \langle w_0,K_{00}v\rangle  \\  \langle v,(K_{00}+1)w_0\rangle  &    \langle v,K_{-10}v\rangle\end{pmatrix} 
\begin{pmatrix} \tilde{\alpha} \\ \beta \end{pmatrix}  = 
\begin{pmatrix} \mathcal{F}_1 \\ \mathcal{F}_2 \end{pmatrix} 
\end{align*}
which we recast as a fixed point problem 
\begin{align}\label{fixedpoint}
\vec{\zeta}=A^{-1} \begin{pmatrix} \mathcal{F}_1 \\ \mathcal{F}_2 \end{pmatrix} =:\vec{F}(\tilde{\alpha},\beta;\eps).
\end{align}
We have computed 
\begin{align*}
A =\begin{pmatrix} 0& 16\\  16&   -32 \end{pmatrix} 
\end{align*}
so in particular, $A$ is invertible. 
We wish to show there is a solution $(\tilde{\alpha},\beta)$ of (\ref{fixedpoint}) of the appropriate size. We establish this fact in the following Lemmas. Lemmas \ref{Kcont} and \ref{W(x)} are accessory to Lemma \ref{alphabeta}.

\begin{lemma}\label{Kcont}
The operators and functions $\overline{K}_2$, $\overline{K}_4$ and $h_1$, $h_2$ are continuous in $\tilde{\alpha}>0$. 
\end{lemma}
\begin{proof}
The operators and function in question are compositions of continuous functions of $\tilde{\alpha}$. 
\end{proof}
\begin{lemma}\label{W(x)}
The $W$ given by Lemma \ref{W} is continuous in $\vec{\zeta}$ for sufficiently small $|\eps|$.
\end{lemma}
\begin{proof}
Let $(\tilde{\alpha}_1,\beta_1)$ give rise to $W_1$ and let $(\tilde{\alpha}_2,\beta_2)$ give rise to $W_2$ via Lemma \ref{W}. Take $|\tilde{\alpha}_1-\tilde{\alpha}_2| <  \delta$ and $|\beta_1-\beta_2|< \delta$. We show that $\|W_1-W_2\|_{L^2} <  C \delta$ for some constant $C> 0$. Observing $\overline{K}_2$ depends on $\tilde{\alpha}$, we see
\begin{align*}
\|&W_1-W_2\|_{L^2}= \|\left(\overline{P}(K_{00}+1)\overline{P} \right)^{-1} \overline{P}_0\|_{L^2 \rightarrow L^2}\| \mathcal{F}(W_1,\vec{\zeta}_1;\eps)-\mathcal{F}(W_2,\vec{\zeta}_2;\eps) \|_{L^2}\\
& \leq C \bigg\|
(\beta_1-\beta_2)K_{00}v + (\tilde{\alpha}_1-\tilde{\alpha}_2) K_{10}w_0 + \eps(\beta_1-\beta_2)\overline{K}_1 v  
  \\
&\quad \quad \quad+\eps \overline{K}_1(W_1-W_2) + \tilde{\alpha}_1 \beta_1 \overline{K}_2(\tilde{\alpha}_1) v  
- \tilde{\alpha}_2 \beta_2 \overline{K}_2(\tilde{\alpha}_2) v
 + \tilde{\alpha}_1 \overline{K}_2(\tilde{\alpha}_1) W_1 
\\
& \quad \quad \quad \quad \quad \quad \quad \ 
- \tilde{\alpha}_2 \overline{K}_2(\tilde{\alpha}_2) W_2
 + \eps(\tilde{\alpha}_1-\tilde{\alpha}_2)f_2 + \tilde{\alpha}_1^2 h_1(\tilde{\alpha}_1) -\tilde{\alpha}_2^2 h_1(\tilde{\alpha}_2)
 \bigg\|_{L^2} \\
& \leq C \delta + C |\eps| \|W_1-W_2\|_{L^2} \\
&\quad \quad + \|\tilde{\alpha}_1 \overline{K}_2(\tilde{\alpha}_1) (W_1-W_2) +\left( \tilde{\alpha}_1 \overline{K}_2(\tilde{\alpha}_1) - \tilde{\alpha}_2 \overline{K}_2(\tilde{\alpha}_2) \right) W_2\|_{L^2} \\
& \leq C \delta + C|\eps| \|W_1-W_2\|_{L^2} 
\end{align*}
noting that $|\tilde{\alpha}_1|\leq M |\eps|^3$. Rearranging the above gives 
\begin{align*}
\|W_1-W_2\|_{L^2} <  C \delta 
\end{align*}
for small enough $|\eps|$. 
\end{proof}

\begin{lemma}\label{alphabeta}
There exists $\eps_0> 0$ such that for all $-\eps_0\leq \eps \leq \eps_0$ \emph{(}$\eps \neq 0$\emph{)} the equation \emph{(\ref{fixedpoint})} has a fixed point with 
$ |\tilde{\alpha}|, |\beta| \leq M|\eps|^3$ for some $M > 0$. 
\end{lemma}
\begin{proof}
We prove this by means of the Brouwer Fixed Point Theorem. We show that $\vec{F}$ maps a closed square into itself and that $\vec{F}$ is a continuous function. Take $ |\tilde{\alpha}|, |\beta| \leq M|\eps|^3$ and and so by  Lemma \ref{W} we have $\|W\|_{L^2} \leq  |\eps|^3R$ for some $R > 0$. Consider now
\begin{align*}
&\| A^{-1}\| \left| \mathcal{F}_1 \right| \\
& \leq \|A^{-1}\| 
 \Bigg( |\eps|  |\beta|    |\langle w_0,\overline{K}_1v\rangle | +  |\tilde{\alpha}|  | \beta|  |\langle w_0,
\overline{K}_2v\rangle |  + |\eps|^3 |\langle w_0,f_1\rangle |  \\
& \quad + |\eps| |\tilde{\alpha}|  |\langle w_0,f_2\rangle | + |\tilde{\alpha}|^2 |\langle w_0,h_1\rangle | + |\eps| |\langle w_0,\overline{K}_1 W\rangle | 
+ |\tilde{\alpha}|  |\langle w_0,\overline{K}_2 W\rangle |  \Bigg)\\
&\leq  CM|\eps|^4 + CM^2 |\eps|^6 + C|\eps|^3  + CM|\eps|^4 + CM^2|\eps|^6 + CR|\eps|^4   \\
&\leq C|\eps|^3 + CM|\eps|^4 \leq M|\eps|^3
\end{align*} 
and
\begin{align*}
&\| A^{-1}\| \left| \mathcal{F}_2 \right| \\
&\leq \| A^{-1}\|  
\Bigg( |\eps|^3 |\langle v, f_4 \rangle| + |\eps|  |\langle v, \overline{K}_3 (\beta v + W)\rangle| + |\tilde{\alpha}| |\eps| |\langle v, f_5\rangle|  \\
& \quad \quad \quad \quad \quad \quad \quad \quad \quad \quad \quad \quad + |\tilde{\alpha}| |\langle v, \overline{K}_4 (\beta v +W)\rangle| + |\tilde{\alpha}|^2  |\langle v, h_2\rangle|
\Bigg) \\
&\leq C|\eps|^3 + CM |\eps|^4 +CR |\eps|^4 + CM |\eps|^4 + CM^2 |\eps|^6 + CMR |\eps|^6 + CM^2 |\eps|^6 \\
&\leq C|\eps|^3 + CM|\eps|^4 \leq M|\eps|^3
\end{align*}
for some choice of $M> 0$ and sufficiently small $|\eps|> 0$. Here $C > 0$ is a constant that is different at each instant. So $\vec{F}$ maps the closed square to itself. 

It is left to show that $\vec{F}$ is continuous. Given $\eta> 0$ take $|\tilde{\alpha}_1 - \tilde{\alpha}_2| <  \delta$ and $|\beta_1 - \beta_2| <  \delta$.
Let $(\tilde{\alpha}_1,\beta_1)$ give rise to $W_1$ and let $(\tilde{\alpha}_2,\beta_2)$ give rise to $W_2$ via Lemma \ref{W}. 
We will also use Lemma \ref{Kcont} and Lemma \ref{W(x)}. 
Now consider 
\begin{align*}
&|\mathcal{F}_1(\tilde{\alpha}_1,\beta_1) - \mathcal{F}_1(\tilde{\alpha}_2,\beta_2) | \\
& = \Big| \eps (\beta_1 -\beta_2) \langle w_0,\overline{K}_1v\rangle  + \tilde{\alpha}_1 \beta_1   \langle w_0, \overline{K}_{2}(\tilde{\alpha}_1)v\rangle   -\tilde{\alpha}_2 \beta_2 \langle w_0, \overline{K}_{2}(\tilde{\alpha}_2)v\rangle  
 \\
& \quad \quad + \eps (\tilde{\alpha}_1-\tilde{\alpha}_2) \langle w_0,f_2\rangle 
+\tilde{\alpha}_1^2 \langle w_0, h_1(\tilde{\alpha}_1) \rangle   - \tilde{\alpha}_2^2 \langle w_0, h_1(\tilde{\alpha}_2) \rangle 
 \\
& \quad \quad + \eps \langle w_0,\overline{K}_1(W_1-W_2) \rangle 
+ \tilde{\alpha}_1 \langle w_0,\overline{K}_2(\tilde{\alpha}_1)W_1 \rangle  -\tilde{\alpha}_2 \langle w_0,\overline{K}_2(\tilde{\alpha}_2)W_2 \rangle \Big|  \\
& \leq C \delta + C \| h_1(\tilde{\alpha}_1)-h_1(\tilde{\alpha}_2)\|_{L^2} \\
&\quad \quad + C \|W_1-W_2\|_{L^2} + C\|\overline{K}_2(\tilde{\alpha}_1)-\overline{K}_2(\tilde{\alpha}_2)\|_{L^2 \rightarrow L^2} \\
&\leq C \delta <  \frac{\eta}{ \|A^{-1}\| \sqrt{2}}
\end{align*} 
for small enough $\delta$. Similarly we can show
\begin{align*}
&|\mathcal{F}_2(\tilde{\alpha}_1,\beta_1) - \mathcal{F}_2(\tilde{\alpha}_2,\beta_2) | 
 \leq C \delta <  \frac{\eta}{ \|A^{-1}\| \sqrt{2}}
\end{align*}
for $\delta$ small enough. Putting everything together gives $|\vec{F}(\vec{\zeta}_1)-\vec{F}(\vec{\zeta}_2)| <  \eta$ as required. Hence $\vec{F}$ is continuous. 
\end{proof}
So finally we have solved both~\eqref{Pbar} and~\eqref{P},
and hence~\eqref{unexp}, and so have proved Theorem~\ref{main}.

\section{Comments on the Computations}\label{secnum}
Analytical and numerical computations were used in the above to compute inner products such as the ones appearing in the definition of $\alpha_2$ (\ref{expansion}). It was critical to establish that $\alpha_2>0$ since the expansion of the resolvent  $R^{(\alpha)}$  (\ref{resolvent}) requires $\alpha>0$. Inner products containing $w_0$ and/or $Pw_1$ but not $w_1$ can be written as an explicit single integral and then evaluated analytically or numerically with good accuracy. For example
\begin{align*}
\langle& w_0, K_{02} w_0 \rangle  + \frac{1}{4}\langle w_0,  K_{00} K_{-12}w_0\rangle  \\
=& -\frac{1}{2} \int_{\R^2}|x-y| \big( 4 Q^2(x) - 3Q^4(x)  \big) \\
& \quad \quad \quad \quad \times \big( Q^2(y) q_1(y)-q_1(y) + 3 Q^2(y) q_2(y) -4 q_2(y) -\frac{c_2}{2} Q^2(y) \big) dy dx \\
&  + \frac{1}{2 \sqrt{2}} \int_{\R^2}e^{-\sqrt{2} |x-y|} \big( 2 Q^2(x) - 3Q^4(x)  \big) \\
& \quad \quad \quad \quad  \times \big( Q^2(y) q_1(y)-q_1(y) + 3 Q^2(y) q_2(y) -2 q_2(y) -\frac{c_2}{4} Q^2(y) \big) dy dx \\
 =&-\int_\R Q^2(y) \big( Q^2(y) q_1(y)-q_1(y) + 3 Q^2(y) q_2(y) -4 q_2(y) -\frac{c_2}{2} Q^2(y) \big) dy \\
&  - \int_\R Q^2(y) \big( Q^2(y) q_1(y)-q_1(y) + 3 Q^2(y) q_2(y) -2 q_2(y) -\frac{c_2}{4} Q^2(y) \big) dy
\\
\approx& -2.9369
\end{align*}
where
\begin{align*}
c_2=\frac{1}{2} \int_\R Q^2 q_1 -q_1 + 3Q^2 q_2 -4q_2. 
\end{align*}
To reduce the double integral to a single integral we recall some facts about the integral kernels. Let 
\begin{align*}
h(y)=-\frac{1}{2} \int_\R |x-y| \big(4Q^2(x) -3Q^4(x) \big)dx.
\end{align*}
Then $h$ solves the equation 
\begin{align*}
h''=-4Q^2 +3Q^4.
\end{align*}
Notice that $-4Q^2 +3Q^4 = -2Q^2 u_1 -Q^2 u_2$ where $u_1$ and $u_2$ are the components of the resonance $u_0$ (\ref{res}). Observing the equation (\ref{syst1}) we see that $h=u_1+c=2-Q^2+c$ for some constant $c$. We can directly compute $h(0)=-2$ to find $c=-2$ and so $h=-Q^2$. A similar argument involving (\ref{syst2}) gives 
\begin{align*}
\frac{1}{2\sqrt{2}} \int_\R e^{-\sqrt{2}|x-y|}\big(2Q^2(x) -3Q^4(x) \big)dx = u_2(y) =-Q^2(y).
\end{align*}

Computing inner products containing $\overline{P} w_1$ is harder. We have an explicit expression for $P w_1$ but lack an explicit expression for $\overline{P} w_1$. Therefore we approximate $\overline{P} w_1$ by numerically inverting $\overline{P}(K_{00}+1)\overline{P}$ in 
\begin{align*}
 \overline{P}(K_{00}+1) \overline{P}w_1&=- \left(\frac{1}{4}\overline{P}K_{00}K_{-11} w_0 + \overline{P}K_{01} w_0\right) =:g.
\end{align*} 
Note that $\langle g,v \rangle = \langle g, w_0 \rangle =0$. We represent $\overline{P}(K_{00}+1) \overline{P}$ as a matrix with respect to a basis $\{\phi_j\}_{j=1}^N$. The basis is formed by taking terms from the typical Fourier basis and projecting out the components of each function in the direction of $v$ and $w_0$. Some basis functions were removed to ensure linear independence of the basis. Let $\overline{P} w_1=\sum_{j=1}^N a_j \phi_j$. Then 
\begin{align*}
B \vec{a}=\vec{b}
\end{align*}
where $B_{j,k}=\langle \phi_j,(K_{00}+1)\phi_k \rangle$ and $b_j=\langle \phi_j, g \rangle$. So we can solve for $\vec{a}$ by inverting the matrix $B$. Once we have an approximation for $\overline{P}w_1$ we can compute $\overline{P}(K_{00}+1)\overline{P}w_1$ directly to observe agreement with the function $g$. 
With an approximation for $\overline{P}w_1$ in hand we can compute inner products containing $\overline{P} w_1$ in the same way as the previous inner product containing $w_0$. In this way we establish that $\alpha_2>0$.

\section*{Acknowledgments}
The authors thank T.P. Tsai for suggesting the problem and for helpful discussions. 
MC is supported by an NSERC CGS.
SG is supported by an NSERC Discovery Grant.

\noindent colesmp@math.ubc.ca

\noindent gustaf@math.ubc.ca


\begin{thebibliography}{99}

\bibitem{Bam}
D. Bambusi,
Asymptotic Stability of Ground States in Some Hamiltonian PDE with Symmetry,
\emph{Comm. Math. Phys.}, Vol. 320, No. 2, (2013), 499-542.  

\bibitem{Busl}
V.S. Buslaev, C. Sulem, 
On Asymptotic Stability of Solitary Waves for Nonlinear Schr\"odinger Equations, 
\emph{Ann. Inst. H. Poincare Anal. Non Lineaire}, Vol. 20, No. 3, (2003), 419-475. 

\bibitem{Caze}
T. Cazenave, 
\emph{Semilinear Schr\"{o}dginer Equations},
American Mathematical Soc., Providence, RI, 2003.

\bibitem{Gust}
S. Chang, S. Gustafson, K. Nakanishi, T. Tsai,
Spectra of Linearized Operators for NLS Solitary Waves,
\emph{SIAM J. Math Anal.},
Vol. 39, No. 4, (2007), 1070-1111.

\bibitem{Cucc1}
S. Cuccagna,
Stabilization of Solutions to Nonlinear Schr\"{o}dinger Equations,
\emph{Comm. Pure Appl. Math.},
54, (2001), 1110-1145. 

\bibitem{Cucc2}
S. Cuccagna, 
On Asymptotic Stability of Ground States of NLS, 
\emph{Rev. Math. Phys.}, 15,
No. 8, (2003), 877-903.

\bibitem{Cucc3}
S. Cuccagna,
On Asymptotic Stability of Moving Ground States of the Nonlinear
Schr\"odinger Equation,
\emph{Trans. Amer. Math. Soc.}, 366, (2014), 2827-2888.  

\bibitem{Dpel}
S. Cuccagna, D. Pelinovsky,
Bifurcations from the Endpoints of the Essential Spectrum in the Linearized Nonlinear Schr\"{o}dinger Problem,
\emph{J. Math. Phys.},
46, (2005), 053520.

\bibitem{Dpel3}
S. Cuccagna, D. Pelinovsky, V. Vougalter, 
Spectra of Positive and Negative Energies in the Linearized NLS Problem, 
\emph{Comm. Pure Appl. Math.},
58, No. 1, (2005), 1-29. 

\bibitem{Dpel2}
S. Cuccagna, D. Pelinovsky,
The Asymptotic Stability of Solitons in the Cubic NLS Equation on the Line,
\emph{Applicable Analysis}, 
93, (2014), 791-822. 

\bibitem{Gang}
Z. Gang, I.M. Sigal, 
Asymptotic Stability of Nonlinear Schr\"odinger Equations with Potential,
\emph{Rev. Math. Phys.}, 17, No. 10, (2005), 1143-1207.

\bibitem{Grillakis}
M. Grillakis,
Linearized instability for nonlinear Schr\"odinger and Klein-Gordon equations,
\emph{Comm. Pure Appl. Anal.}, 41, (1988), 747-774.

\bibitem{Gril}
M. Grillakis, J. Shatah, W. Strauss, 
Stability Theory of Solitary Waves in the Presence of Symmetry I, 
\emph{J. Funct. Anal.}, 74, (1987), 160-197.

\bibitem{Book}
S. Gustafson, I.M. Sigal, 
\emph{Mathematical Concepts of Quantum Mechanics} (2nd ed.), Springer-Verlag Berlin Heidelberg, 2011. 

\bibitem{Fibich}
G. Fibich,
\emph{The Nonlinear Schr\"odinger Equation: Singular Solutions and Optical Collapse},
Springer, 2015.

\bibitem{Kato}
A. Jensen, T. Kato,
Spectral Properties of  Schr\"{o}dinger Operators and Time-Decay of the Wave Functions,
\emph{Duke Math. J.}, Vol. 46, No. 3, (1979), 583-611.

\bibitem{Jens}
A. Jensen, G. Nenciu,
A Unified Approach to Resolvent Expansions at Thresholds,
\emph{Rev. Math. Phys.},
13, (2001), 717.

\bibitem{Evan1}
T. Kapitula, B. Sandstede,
Edge Bifurcations for Near Integrable Systems via Evans Functions,
\emph{SIAM J. Math Anal.},
Vol. 33, No. 5, (2002), 1117-1143.

\bibitem{Evan2}
T. Kapitula, B. Sandstede,
Eigenvalues and Resonances Using the Evans Functions,
\emph{Discrete Contin. Dyn. Syst.}, Vol. 10, No. 4,  (2004), 857-869.

\bibitem{Pere}
G. Perelman,
Asymptotic Stability of Multi-Soliton Solutions for Nonlinear Schr\"{o}dinger Equations,
\emph{Comm. Partial Differential Equations}, 
29, No. 7-8, (2004), 1051-1095. 

\bibitem{Schl}
W. Schlag, 
Stabile Manifolds for an Orbitally Unstable Nonlinear Schr\"{o}dinger Equation,
\emph{Ann. of Math.}, 169 , No. 1, (2009), 139-227. 

\bibitem{Sulem}
C. Sulem, P.-L. Sulem, 
\emph{The Nonlinear Schr\"{o}dinger Equation},
Springer, 1999.

\bibitem{Voug}
V. Vougalter,
On Threshold Eigenvalues and Resonances for the Linearized NLS Equation,
\emph{Math. Model. Nat. Phenom.},
Vol. 5, No. 4, (2010), 448-469. 

\bibitem{Voug2}
V. Vougalter,
On the Negative Index Theorem for the Linearized NLS Problem, 
\emph{Canad. Math. Bull.}, 
53, (2010), 737-745. 

\bibitem{Voug3}
V. Vougalter, D. Pelinovsky,
Eigenvalues of Zero Energy in the Linearized NLS Problem, 
\emph{Journal of Mathematical Physics},
47, (2006), 062701. 

\bibitem{Wein1}
M. I. Weinstein, 
Modulational Stability of Ground States of Nonlinear Schr\"{o}dinger Equations,
\emph{SIAM J. Math Anal.},
16, (1985), 472-491. 

\bibitem{Wein2}
M. I. Weinstein, 
Lyapunov Stability of Ground States of Nonlinear Dispersive Evolutions Equations, 
\emph{Comm. Pure Appl. Math.}, 
39, (1986), 51-68. 

\end{thebibliography}
\end{document}